\documentclass[a4paper, 11pt,reqno]{amsart}

\usepackage{amsmath}
\usepackage{amssymb}
\usepackage{hyperref} 
\usepackage{amsfonts}
\usepackage{mathtools}
\usepackage{theoremref}
\usepackage{amsthm}
\usepackage[dvipsnames]{xcolor}
\usepackage[utf8]{inputenc}
\usepackage{array}
\usepackage{enumitem}
\usepackage[english]{babel}
\usepackage{comment}

\usepackage{tikz}
\usetikzlibrary{shapes.geometric}
\usetikzlibrary {arrows.meta}

\calclayout

\setlist{topsep=0mm,partopsep=0mm,itemsep=1mm}

\makeatletter
\@namedef{subjclassname@2010}{\textup{2020} Mathematics Subject Classification}
\makeatother

\theoremstyle{plain}
\newtheorem{lemma}{Lemma}[section]

\newtheorem{prop}[lemma]{Proposition}
\newtheorem*{mainthm}{Main Theorem}
\newtheorem*{thm*}{Theorem}

\theoremstyle{definition}
\newtheorem{cons}[lemma]{Construction}
\newtheorem{que}[lemma]{Question}

\theoremstyle{remark}

\newtheorem{cclaim}{Claim}[lemma]

\newenvironment{claim}{\begin{cclaim}\it}{\end{cclaim}}

\newcommand{\N}{\mathbb{N}}
\newcommand{\Z}{\mathbb{Z}}

\newenvironment{nitemize}{\begin{itemize}[label=\textbullet, leftmargin=5mm]}{\end{itemize}}

\begin{document}

\title[Subdirect powers of semigroups of integers]{On the number of subdirect products involving semigroups of integers and natural numbers}
\author[A.\ Clayton]{Ashley Clayton}
\address{School of Mathematics and Statistics, University of St Andrews, St Andrews, Scotland, UK}
\email{ac323@st-andrews.ac.uk}

\author[C.\ Reilly]{Catherine Reilly}
\address{School of Mathematics, University of East Anglia, Norwich NR4 7TJ, England, UK}
\email{C.Reilly@uea.ac.uk}

\author[N.\ Ru\v{s}kuc]{Nik Ru\v{s}kuc}
\address{School of Mathematics and Statistics, University of St Andrews, St Andrews, Scotland, UK}
\email{nr1@st-andrews.ac.uk}


\begin{abstract}
We extend a recent result 
that for the (additive)   semigroup of positive integers $\N$, there are
continuum 
many subdirect products of $\N \times \N$ up to isomorphism. 
We prove that for $U,V$ each one of $\Z$ (the group of integers), $\N_{0}$ (the monoid of non-negative integers), or $\N$, the direct product $U \times V$ contains continuum many (semigroup) subdirect products up to isomorphism. 
\end{abstract}

\thanks{The third author acknowledges support from EPSRC EP/V003224/1.}
\maketitle

\section{Introduction}

In \cite{acnr1} it is proved that the direct product $\N\times\N$ of two copies of the free monogenic semigroup $\N$ contains uncountably many pairwise non-isomorphic subdirect products.
This is perhaps somewhat surprising, given that the direct product $\Z\times\Z$ of two copies of the free cyclic group contains only two subdirect products up to isomorphism, namely $\Z$ and $\Z\times\Z$ itself, and that the subsemigroup structure of $\N$ is not fundamentally different from the subgroup structure of $\Z$, in that both essentially depend on arithmetic progressions; see \cite{sitsiu75} for an explicit description.

The purpose of this paper is to extend the scope of the above-mentioned result from \cite{acnr1} and prove the following:

\begin{mainthm}
Let each of $U$ and $V$ be any of the following three additive semigroups: $\Z$, the group of integers; $\N_0$, the monoid of non-negative integers; $\N$, the semigroup of natural numbers. Then $U\times V$ contains continuum many non-isomorphic semigroup subdirect products of $U$ and~$V$. 
\end{mainthm}

By a \emph{subdirect product} of two semigroups $U$ and $V$ we mean any subsemigroup $P$ of $U\times V$ which projects \emph{onto} each of $U$ and $V$, i.e. $\{ u\::\: (u,v)\in P \text{ for some } v\}=U$ and
$\{ v\::\: (u,v)\in P \text{ for some } u\}=V$.
Subdirect products are an important decomposition tool in algebra in general, due to Birkhoff's decomposition theorem \cite[Theorem 4.44]{mmt87}.
They also have many intriguing combinatorial properties.
For some examples from group theory see \cite{baum00, baum84, brid09}, and for a discussion from the viewpoint of general algebra see \cite{mayr19}.

The rest of the paper constitutes the proof of the Main Theorem, using the following outline.
For reasons of symmetry, and keeping in mind that the case where $U=V=\N$ has been dealt with in \cite{acnr1}, it is sufficient to prove the theorem for $(U,V)$ in 
$\mathcal{P}=\{ (\N_0,\N), (\N_0,\N_0), (\Z,\N),(\Z,\N_0), (\Z,\Z) \}$.
In Section \ref{sec:cons} we construct a  family of subsemigroups $S_\sigma$ of $\Z \times \Z$, where $\sigma$ is a sequence of natural numbers with certain additional requirements. 
These requirements are sufficiently mild that the number of sequences satisfying them is uncountable.
We begin Section \ref{sec:inter} by proving that the intersection $S_\sigma\cap (U\times V)$ 
is a subdirect product  in $U\times V$ for each $(U,V)\in\mathcal{P}$ (Lemma \ref{la:insub}).
In the remainder of Section \ref{sec:inter} we consider each possibility for $(U,V)$ in turn, starting with $(U,V)=(\N_0,\N)$,
and show that
for $\sigma\neq\tau$ we have  $S_\sigma\cap (U\times V)\ncong S_\tau\cap (U\times V)$.
Thus the subsemigroups $S_\sigma\cap (U\times V)$ constitute uncountably many pairwise non-isomorphic subdirect products in $U\times V$, and the Main Theorem is proved.

Of the several assertions encompassed by the Main Theorem, perhaps the one concerning $\Z\times \Z$
is worth highlighting as somewhat surprising. As mentioned earlier, $\Z\times\Z$ contains \emph{countably many group} subdirect products. However, our result shows that it contains \emph{uncountably many semigroup} subdirect products.

\section{The semigroups $S_\sigma$}
\label{sec:cons}

We begin our work towards proving the Main Theorem by exhibiting a family $S_\sigma$ of subdirect products of $\Z\times\Z$ indexed by certain infinite sequences of natural numbers.
We first define the sets $S_\sigma\subseteq \Z\times \Z$, then prove they are subsemigroups, and finally that they are subdirect products.

\begin{cons}\label{cons:ssigma}
Given a sequence, $\sigma = (c_i)_{i\geq 2}$ of natural numbers satisfying 
\begin{equation}
\label{eq:sigma}
c_2=1\quad \text{and}\quad c_{i+1}\geq 2c_i \text{ for all } i\geq 2,
\end{equation}
define 
\[S_\sigma := \{(x, y) : x \leq 0, y\geq x\} \cup  \bigcup\limits_{k=2}^{\infty} \{(x, x + k) :  x = 1, \dots, c_k \}. \]
\end{cons}

The following comments and Figure \ref{fig:Ssig} may be of help in understanding $S_{\sigma}$ and how it will be treated subsequently.
\begin{nitemize}
\item
It is useful to consider $S_\sigma$ as a union of `vertical lines'. Specifically, $S_\sigma=\bigcup_{i\in \Z} L_i$, where
$L_i:=S_\sigma\cap ( \{i\} \times \Z)$.
\item
The lines $L_i$ with $i\leq 0$ are the same for all $S_\sigma$,  namely $L_i=\{ (i,x)\::\: x\geq i\}$.
\item
The remaining lines $L_i$, $i>0$, depend on $\sigma$. Each such line $L_i$ has a unique `lowest point', denoted $(i,l_i)$. The construction assures that $l_i>i$. The line contains all points above this lowest point, meaning that $(i,x) \in L_{i}$ for all integers $x \geq l_{i}$.
\item
The number $c_k = i$ indicates the rightmost line $L_i$ for which the lowest point is $(i,i+k)$.
\item
In other words, for any $i>0$ and $k\geq 2$, we have $L_i=\{ (i,x)\::\: x\geq i+k\}$ if and only if $c_{k-1}<i\leq c_k$ for all $i,k>0$.
\item
The conditions $c_2=1$ and $c_{i+1} \geq 2c_i$ are technical, and are needed to facilitate the proofs of closure below and non-isomorphism later on.
\item
Due to the fixed requirement $c_2=1$, we have that $L_1=\{ (1,x)\::\: x\geq 3\}$ is still the same for all $S_\sigma$.
\end{nitemize}

The above terminology and notation will be used throughout the paper. In Figure \ref{fig:Ssig} we visualise a typical example of $S_{\sigma}$.

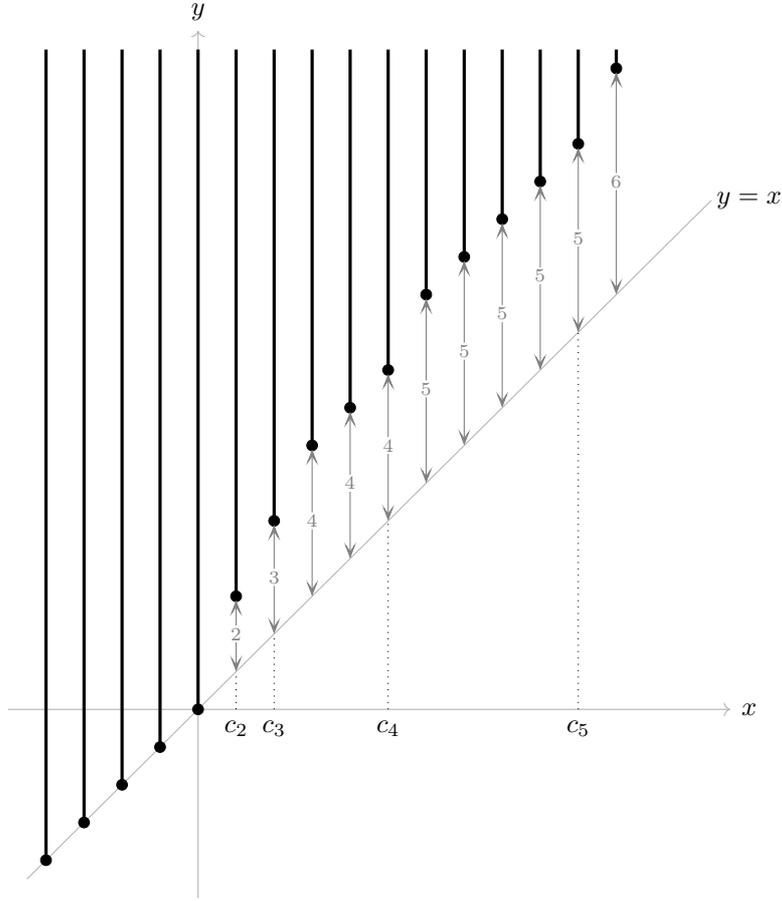
\begin{figure}[h!]
\begin{center}
\begin{tikzpicture}[x=5mm,y=5mm]
\draw [black!30,->] (-5,0) -- (14,0) {};
\draw [black!30,->] (0,-5) -- (0,18) {};
\draw [black!30] (-4.5,-4.5) -- (13.5,13.5) {};

\node at (14.5,0) {\small $x$};
\node at (0,18.5) {\small $y$};
\node at (14.5,13.5) {\small $y=x$};

\filldraw (-4, -4) circle (2pt);
\draw[very thick] (-4,-4) -- (-4,17.5);
\filldraw (-3, -3) circle (2pt);
\draw[very thick] (-3,-3) -- (-3,17.5);
\filldraw (-2, -2) circle (2pt);
\draw[very thick] (-2,-2) -- (-2,17.5);
\filldraw (-1, -1) circle (2pt);
\draw[very thick] (-1,-1) -- (-1,17.5);
\filldraw (0, 0) circle (2pt);
\draw[very thick] (0,0) -- (0,17.5);
\filldraw (1, 3) circle (2pt);
\draw[very thick] (1,3) -- (1,17.5);
\filldraw (2, 5) circle (2pt);
\draw[very thick] (2,5) -- (2,17.5);
\filldraw (3, 7) circle (2pt);
\draw[very thick] (3,7) -- (3,17.5);
\filldraw (4, 8) circle (2pt);
\draw[very thick] (4,8) -- (4,17.5);
\filldraw (5, 9) circle (2pt);
\draw[very thick] (5,9) -- (5,17.5);
\filldraw (6, 11) circle (2pt);
\draw[very thick] (6,11) -- (6,17.5);
\filldraw (7, 12) circle (2pt);
\draw[very thick] (7,12) -- (7,17.5);
\filldraw (8, 13) circle (2pt);
\draw[very thick] (8,13) -- (8,17.5);
\filldraw (9, 14) circle (2pt);
\draw[very thick] (9,14) -- (9,17.5);
\filldraw (10, 15) circle (2pt);
\draw[very thick] (10,15) -- (10,17.5);
\filldraw (11, 17) circle (2pt);
\draw[very thick] (11,17) -- (11,17.5);

\path [{Stealth[inset=2.6pt,scale=1.3]}-{Stealth[inset=2.6pt,scale=1.3]}, black!50,draw] (1,1) --
(1,2) node [circle,fill=white,inner sep=0,draw=white] {{\tiny 2}} -- (1,2.9);
 \draw [{Stealth[inset=2.6pt,scale=1.3]}-{Stealth[inset=2.6pt,scale=1.3]},black!50] (2,2) -- (2,3.5) node [circle,fill=white,inner sep=0,draw=white] {{\tiny 3}} -- (2,4.9);
 \draw [{Stealth[inset=2.6pt,scale=1.3]}-{Stealth[inset=2.6pt,scale=1.3]},black!50] (3,3) -- (3,5) node [circle,fill=white,inner sep=0,draw=white] {{\tiny 4}} -- (3,6.9);
  \draw [{Stealth[inset=2.6pt,scale=1.3]}-{Stealth[inset=2.6pt,scale=1.3]},black!50] (4,4) -- (4,6) node [circle,fill=white,inner sep=0,draw=white] {{\tiny 4}} -- (4,7.9);
\draw [{Stealth[inset=2.6pt,scale=1.3]}-{Stealth[inset=2.6pt,scale=1.3]},black!50] (5,5) -- (5,7) node [circle,fill=white,inner sep=0,draw=white] {{\tiny 4}} --(5,8.9);
\draw [{Stealth[inset=2.6pt,scale=1.3]}-{Stealth[inset=2.6pt,scale=1.3]},black!50] (6,6) -- (6,8.5) node [circle,fill=white,inner sep=0,draw=white] {{\tiny 5}} --(6,10.9);
\draw [{Stealth[inset=2.6pt,scale=1.3]}-{Stealth[inset=2.6pt,scale=1.3]},black!50] (7,7) -- (7,9.5) node [circle,fill=white,inner sep=0,draw=white] {{\tiny 5}} --(7,11.9);
\draw [{Stealth[inset=2.6pt,scale=1.3]}-{Stealth[inset=2.6pt,scale=1.3]},black!50] (8,8) -- (8,10.5) node [circle,fill=white,inner sep=0,draw=white] {{\tiny 5}} --(8,12.9);
\draw [{Stealth[inset=2.6pt,scale=1.3]}-{Stealth[inset=2.6pt,scale=1.3]},black!50] (9,9) -- (9,11.5) node [circle,fill=white,inner sep=0,draw=white] {{\tiny 5}} --(9,13.9);
\draw [{Stealth[inset=2.6pt,scale=1.3]}-{Stealth[inset=2.6pt,scale=1.3]},black!50] (10,10) -- (10,12.5) node [circle,fill=white,inner sep=0,draw=white] {{\tiny 5}} -- (10,14.9);
\draw [{Stealth[inset=2.6pt,scale=1.3]}-{Stealth[inset=2.6pt,scale=1.3]},black!50] (11,11) -- (11,14) node [circle,fill=white,inner sep=0,draw=white] {{\tiny 6}} -- (11,16.9);

\draw [dotted] (1,0)--(1,1);
\draw [dotted] (2,0)--(2,2);
\draw [dotted] (5,0)--(5,5);
\draw [dotted] (10,0)--(10,10);

\node at (1,-0.5) {{\small $c_2$}};
\node at (2,-0.5) {{\small $c_3$}};
\node at (5,-0.5) {{\small $c_4$}};
\node at (10,-0.5) {{\small $c_5$}};

\end{tikzpicture}
\caption{The semigroup $S_\sigma$, with $\sigma=(1,2,5,10,\dots)$.}
\label{fig:Ssig}
\end{center}
\end{figure}

\begin{lemma}
\label{la:subs}
Each $S_\sigma$ is a subsemigroup of $\mathbb{Z}\times\mathbb{Z}$.
\begin{proof} We show that $S_\sigma$ is closed under pairwise addition. To this end, let
$\mu, \nu \in S_\sigma$, with $\mu = (p, q)$, $\nu = (r, s)$. Without loss of generality we may suppose that $p \leq r$. We split the proof  that $\mu + \nu = (p+r, q+s) \in S_\sigma$ into cases, depending on the sign of $p + r$.

{\it Case 1:} $p + r \leq 0$.
In this instance, to show $(p+r,q+s) \in S_{\sigma}$, it suffices to show that $q+s \geq p+r$. This follows, as $(p,q), (r,s) \in S_{\sigma}$ implies $q \geq p$ and $s \geq r$ by construction, whence $q+s \geq p+r$.

{\it Case 2:} $p + r > 0$. 
From $p\leq r$ we have $r>0$. Let $k,l\geq 2$ be the unique numbers such that
\begin{align}
\label{eq:e1}
&c_{k-1}<r\leq c_k,\\
\label{eq:e2}
&c_{l-1}<p+r\leq c_l.
\end{align}
To show that $(p+r,q+s)\in S_\sigma$ it suffices to show that $q+s \geq p + r + l$.

If $p\leq 0$ then $p+r\leq r$, $c_l\leq c_k$ and $l\leq k$ follow in order, and then
\[q+s \geq p+r +k \geq p+r+l.\]
Suppose now that $p>0$.
Let $j\geq 2$ be the unique number such that 
\[c_{j-1} < p \leq c_{j},\]
whereby $q \geq p +j$.
We have that
\[q+s \geq p+r+j+k,\]
from which it follows that
\[
c_{l-1} \leq p+r \leq 2r\leq 2c_{k} \leq c_{k+1} \leq c_{j+k-1},
\]
using \eqref{eq:sigma}-\eqref{eq:e2} and $j-1 \geq 1$. This implies $l \leq j+k$, and so
\[q+s \geq p+r+l\]
as required, completing the proof  that $S_\sigma\leq \Z\times\Z$.\end{proof}\end{lemma}

\begin{lemma}
\label{la:subd}
Each $S_\sigma$ is a subdirect product of $\Z\times\Z$.
\end{lemma}

\begin{proof}
Any integer can be obtained as the first coordinate of a pair using the elements $(1, 3)$ and $(-1, -1)$, which are in $S_{\sigma}$ for every $\sigma$. The same can be done in the second coordinate using $(0, 1),(-1, -1)\in S_\sigma$. 
\end{proof}

\section{Intersection of $S_\sigma$ with some subsemigroups of $\Z\times\Z$}
\label{sec:inter}

In this section, let $(U,V)\in\{ (\Z,\Z),(\Z,\N_0),(\Z,\N),(\N_0,\N_0),(\N_0,\N)\}$. Recall from the introduction that we need only consider such $(U,V)$ to prove our Main Theorem.

 Having constructed the semigroups $S_\sigma$ in the preceding section as subsemigroups of $\mathbb{Z} \times \mathbb{Z}$, this gives us the following way of obtaining subdirect products of $U \times V$ from them.

\begin{lemma}
\label{la:insub}
The intersection $S_{\sigma} \cap (U \times V)$ is a subdirect product of $U \times V$.
\end{lemma}

\begin{proof}
First, the intersection is a subsemigroup of $U \times V$, as $U\times V$ and $S_{\sigma}$ are both subsemigroups of $\Z \times \Z$ (the latter by Lemma \ref{la:subs}). 

It then just remains to show that the projection maps onto $U$ and $V$ are surjective.
 For any $i \in U$ the line $L_{i}$ has non-empty intersection with $S_{\sigma} \cap (U \times V)$, and any element of this line has first coordinate $i$. This gives surjectivity of the first projection map.

For the second projection map, if $j \in V$ is such that $j < 0$, it must be that $U = V = \Z$, in which case $S_{\sigma} \cap (U \times V) = S_{\sigma}$, which is a subdirect product by Lemma \ref{la:subd}. 
 
If $j = 0$, then $V$ is one of $\Z$ or $\N_{0}$, and we have $(0,0) \in S_\sigma \cap (U \times V)$.

Finally, if $j > 0$, then as $L_{0} \setminus \{(0,0)\}\subseteq S_{\sigma} \cap (U \times V)$, it follows that $(0,j)\in S_{\sigma} \cap (U \times V)$.

This completes the proof of surjectivity of the second projection map, and thus of the lemma.
\end{proof}

If we can show that different sequences $\sigma$ and $\tau$ give non-isomorphic subdirect products $S_{\sigma} \cap (U \times V)$ and  $S_{\tau} \cap (U \times V)$, this will be sufficient to prove our Main Theorem.

In the following subsections, we will use the notion of \emph{indecomposability}. 
In fact, we will use this term in two different senses. 
Suppose $W$ is a subsemigroup of $\Z\times\Z$. An element $(a,b)\in W$ is \emph{semigroup indecomposable} if it cannot be written as the sum of any two elements from $W$. In case where $W$ is a monoid, i.e. where $W$ contains the element $(0,0)$, we say that $(a,b)\in W$ is \emph{monoid indecomposable} if it cannot be written as the sum of any two elements of $W\setminus\{(0,0)\}$.
Typically, we will omit the adjective `semigroup' or `monoid' when talking about indecomposability, but it will always be clear from context which one is meant.

\subsection{Intersection with \boldmath{$\N_{0} \times \N$}}
\label{subs:1}
We will start with the case where $U \times V = \N_{0}\times \N$. The semigroup $ S_{\sigma} \cap (\N_{0} \times \N) $ is  just the union of the lines $\{L_{i} : i \geq 0\}$ from $S_{\sigma}$, but without the identity $(0,0)$. Recall that the lowest point of a line $L_{i}$ is denoted $(i,l_{i})$.

We describe the indecomposables of $S_{\sigma} \cap (\N_{0} \times \N)$ in the following lemma, which will be useful in ruling out possible isomorphisms between these semigroups.

\begin{lemma}\thlabel{lem:indecn0xn}
The set of indecomposable elements of $S_{\sigma} \cap (\N_{0} \times \N)$ is exactly the set
\[ \{(0,1)\}\cup\{(i,l_{i}) : i \geq 1 \}.\]
\end{lemma}

\begin{proof}
As
\[(i,j) = (i,l_{i}) + (j-l_{i})(0,1)\]
 for all $(i,j) \in  S_{\sigma} \cap (\N_{0} \times \N) $, then any element which is not the lowest point of its line is decomposable. Hence it remains to show that $(0,1)$ and the lowest points of each line are indecomposable in $ S_{\sigma} \cap (\N_{0} \times \N) $.

Firstly, $(0,1)$ is indecomposable as $1$ is indecomposable in $\N$. 

Now suppose that some element $(i, l_i)$ for $i\in\N$ is decomposable, say
\begin{equation}\label{eq:2.2}
(i, l_i) = (j, q) + (k, r)
\end{equation}
for some $(j, q), (k, r) \in S_{\sigma}\cap(\N_0\times\N)$.

Note that we cannot have $j = 0$ or $k = 0$ as that would contradict $(i, l_i)$ being the lowest point of the line $L_i$; thus $j, k \geq 1$. Now let $x, y, z \geq 2$ be the smallest possible satisfying
\begin{enumerate}[label=(\roman*)]
\item\label{2.2list1} $i \leq c_x$, so that $l_i = i + x$;
\item\label{2.2list2} $j \leq c_y$, so that $l_j = j + y \leq q$;
\item\label{2.2list3} $k \leq c_z$, so that $l_k = k + z \leq r$.
\end{enumerate}

From, (\ref{eq:2.2}), \ref{2.2list1}, \ref{2.2list2} and \ref{2.2list3}, we have:
\[ j+k+x = i+x = l_i = q+r \geq l_j + l_k = j+k+y+z,\]
and hence $x \geq y + z$. Recalling that $c_{n+1} \geq 2c_n$ for all $n \geq 2$, we have that $c_{y+z-1} \geq 2^{z-1}c_y$ and $c_{y+z-1} \geq 2^{y-1}c_z$. Using this, together with $y, z \geq 2$ and \ref{2.2list2} and \ref{2.2list3}, we have:
\[ i = j+k \leq c_y + c_z \leq \bigl(\frac{1}{2^{z-1}} + \frac{1}{2^{y-1}}\bigr) c_{y+z-1} \leq c_{y+z-1} \leq c_{x-1} \]
a contradiction with minimality of $x$ with respect to \ref{2.2list1}. Hence, the elements of the form $(i, l_i)$ are all indecomposable.
\end{proof}

We can now prove the main result of this section -- that there are continuum many subdirect products of $\N_{0} \times \N$ up to isomorphism.

\begin{prop}\label{prop:nonisomn0xn} For any two sequences $\sigma$ and $\tau$ satisfying the conditions of Construction \ref{cons:ssigma}, we have that
\[\sigma \not = \tau \,\Rightarrow\,  S_{\sigma} \cap (\N_{0} \times \N) \not \cong  S_{\tau} \cap (\N_{0} \times \N).\]
Consequently, there are continuum many subdirect products of $\N_{0} \times \N$ up to isomorphism.
\end{prop}

\begin{proof}
We will prove the contrapositive. So suppose two subdirect products $ S_{\sigma} \cap (\N_{0} \times \N) $ and $ S_{\tau} \cap (\N_{0} \times \N) $ are isomorphic, and let $\varphi$ be an isomorphism between them.

This isomorphism must map the indecomposable elements of $ S_{\sigma} \cap (\N_{0} \times \N) $ bijectively onto indecomposable elements of $ S_{\tau} \cap (\N_{0} \times \N) $.

Any indecomposable $(i,l_{i})$ in either semigroup has the property that $(i,l_{i})+(i,l_{i})$ has more than one decomposition into a sum of indecomposable elements, as
\[(i,l_{i}) + (i,l_{i}) = (2i,l_{2i}) + (2l_{i} - l_{2i})(0,1). \] 
By way of contrast,  $(0,1) + (0,1)$ has only that one decomposition into a sum of indecomposables. Hence it must be that $\varphi(0,1) = (0,1)$.

Now consider the image of the indecomposable $(1,3)$, say $\varphi(1,3) = (j,l_{j})$ for some $j \in \N$. Then for any $n \in \N$, it must be that 
\[(nj,nl_{j}) = \varphi(n,3n) = \varphi\bigl((n, l_n) + (3n - l_n)(0,1)\bigr) = \varphi(n, l_n) + (0,3n-l_n). \]

It follows that $\varphi(n,l_n)$ belongs to the line $L_{nj}$.
Furthermore, since it must be indecomposable, we have
 \begin{equation}\varphi(n,l_{n}) = (nj,l_{nj}).\label{eq:indecompim}\end{equation} 
 For $\varphi$ to be surjective on the set of indecomposables, it must be that $j=1$. 
It follows that $\varphi$ is
 the identity mapping, 
 since $S_\sigma\cap (\N_0\times \N)$ is generated by its indecomposable elements. 
Therefore $S_{\sigma} \cap (\N_{0} \times \N) = S_{\tau} \cap (\N_{0} \times \N)$, and hence $\sigma = \tau$, proving the result.\end{proof}

\subsection{Intersection with \boldmath{$\N_{0} \times \N_{0}$}}
\label{subs:2}
We now consider the case where $U\times V = \N_{0}\times\N_{0}$. The semigroup $S_{\sigma} \cap (\N_{0} \times \N_{0})$ is the union of lines $\{L_{i} : i \geq 0\} $ from $S_{\sigma}$. In fact, these semigroups are simply the semigroups $S_{\sigma} \cap (\N_0\times\N)$ with the identity element $(0,0)$ adjoined. Therefore, as an immediate consequence of Proposition \ref{prop:nonisomn0xn} we have
\begin{prop}
 $\N_0 \times \N_0$ has continuum many subdirect products up to isomorphism.
\end{prop}

\subsection{Intersection with \boldmath{$\Z \times \N$}}
\label{subs:3}
Considering the case where $U\times V = \Z \times \N$, we have
\[
S_\sigma\cap (\Z\times\N)=\{ (i,j)\::\: i\leq 0,\ j\geq 1\} \cup \bigcup_{i\geq 1} L_i.
\]

We describe the indecomposable elements of $S_\sigma \cap (\Z \times \N)$ in the following lemma, which is again used to rule out non-identity
isomorphisms between these semigroups.

\begin{lemma}\label{lem:indeczxn}
The set of indecomposable elements of $S_\sigma \cap (\Z\times\N)$ is exactly the set\[ \{(i,l_{i}) : i \geq 1 \}\cup\{(i, 1) : i \leq 0\}.\]
\end{lemma}

\begin{proof} 
Notice that
\[
(i,j)=
\begin{cases} (i, l_i) + (j - l_i)(0, 1) & \text{when } i\geq 1,\ j>l_i\\
(i, 1) + (j - 1)(0, 1)& \text{when } i\leq 0,\ j>1.
\end{cases}
\]
Thus all of these elements are decomposable.

The elements $(i, 1)$  for $i \leq 0$ are indecomposable in $\Z\times\N$ , as $1$ is indecomposable in $\N$. 
It remains to consider $(i,l_{i})$  where $i \geq 1$. Suppose that $(i, l_i)$ is decomposable, say $(i, l_i) = (a, x) + (b, y)$. We cannot have $a, b \geq 0$ by \thref{lem:indecn0xn}. Without loss of generality, suppose $a < 0$. Then $b = i - a > i$, and hence \[y + x > y \geq l_b = b + c_b > i + c_i = l_i,\]
a contradiction.
\end{proof}

We can now move on to proving that there are continuum many subdirect products of $\Z\times\N$ up to isomorphism.

\begin{prop}
\label{prop:zxn}
For any two sequences $\sigma$ and $\tau$ satisfying the conditions of  Construction \ref{cons:ssigma}, we have that 
\[\sigma \not = \tau \,\Rightarrow\,  S_{\sigma} \cap (\Z \times \N) \not \cong  S_{\tau} \cap (\Z \times \N).\]
Consequently, there are continuum many subdirect products of $\Z\times\N$ up to isomorphism.
\end{prop}

\begin{proof} Suppose that $\varphi : S_{\sigma} \cap (\Z\times\N) \rightarrow S_{\tau} \cap (\Z\times\N)$ is an isomorphism. We proceed via a sequence of claims, aiming to show that $\varphi(S_{\sigma} \cap (\N_0 \times \N)) = S_{\tau} \cap (\N_0 \times \N)$ and then use Proposition \ref{prop:nonisomn0xn} to obtain $\sigma=\tau$.

\begin{claim}
\label{cl:zxn1}
$\varphi(0, 1) = (0, 1)$. 
\end{claim}

\begin{proof}
We claim that $(0,1)$ is the only indecomposable element $(x,y)$ such that $(x,y)+(x,y)$ cannot be expressed as a sum of indecomposables in any other way, and the assertion then follows from this.
That $(0,1)$ has this property follows from Lemma \ref{lem:indeczxn}.
For any other indecomposable we have
\[
(i,l_i)+(i,l_i)=(2i,l_{2i})+(2l_i-l_{2i})(0,1),
\]
an alternative decomposition as a sum of indecomposables.
\end{proof}

\begin{claim}\label{cl:zxn2} $\varphi(i,l_{i}) \in L_{p}$ for each $i \geq 1$, where $p$ is $i$ times the first coordinate of $\varphi(1,3)$.\end{claim}

\begin{proof}
As $\varphi(i,3i) = \varphi((i,l_{i}) + (3i-l_{i})(0,1))$, then
\[i\varphi(1,3) = \varphi(i,l_{i}) + (3i-l_{i})(0,1)\]
by Claim \ref{cl:zxn1}, and hence $ \varphi(i,l_{i})$ and $i\varphi(1,3)$ must have the same first coordinate.
\end{proof}

\begin{claim}\label{cl:zxn3} $\varphi(-i,1) \in L_{q}$ for each $i \geq 0$, where $q$ is $i$ times the first coordinate of $\varphi(-1,1)$.\end{claim}

\begin{proof}
As $\varphi(-i,i) = \varphi((-i,1) + (i-1)(0,1))$, then
\[i\varphi(-1,1) = \varphi(-i,1) + (i-1)(0,1)\]
by Claim \ref{cl:zxn1}, and hence $ \varphi(-i,1)$ and $i\varphi(-1,1)$ must have the same first coordinate.\end{proof}

\begin{claim}\label{cl:zxn4} $\varphi(1,3),\varphi(-1,1) \in L_{-1} \cup L_{1}$, and hence either 
\begin{alignat*}{3}
&\varphi(1,3) = (1,3)&\quad&\text{and}&\quad& \varphi(-1,1) = (-1,1); \text{or} \\
&\varphi(1,3) = (-1,1)&&\text{and}&& \varphi(-1,1) = (1,3).
\end{alignat*}
\end{claim}
\begin{proof}Suppose $\varphi(1,3) \in L_{m}$ and $\varphi(-1,1) \in L_{n}$  for some $m,n \in \Z$. Then by Claim \ref{cl:zxn2} and Claim \ref{cl:zxn3}, it would follow that  $\varphi(i,l_{i}) \in L_{mi}$  and  $\varphi(-i,1) \in L_{in}$ for all $i \geq 0$. 

As the set of indecomposables of $S_{\sigma} \cap (\Z \times \N)$ must map bijectively to the set of indecomposables of  $S_{\tau} \cap (\Z \times \N)$, then by Lemma \ref{lem:indeczxn}, it must be that $\{m,n\} = \{-1,1\}$ for $\varphi$ to be surjective.

The last part of the claim follows as either $\varphi(1,3) \in L_{1}$, $\varphi(-1,1) \in L_{-1}$ or $\varphi(1,3) \in L_{-1}$, $\varphi(-1,1) \in L_{1}$, and noting that each of $(1,3)$, $(-1,1)$ must map to the unique indecomposable of the given line.
\end{proof}

\begin{claim} \label{cl:zxn5}$\varphi(1,3) = (1,3)$, $\varphi(-1,1) = (-1,1)$. \end{claim}

\begin{proof}
Suppose otherwise, which by Claim \ref{cl:zxn4} would force $\varphi(1,3) = (-1,1)$, $\varphi(-1,1) = (1,3)$. On one hand, for $n \in \N$, we have
\begin{equation}\label{-nnoneway}\varphi(-n,n) = \varphi(n(-1,1)) = (n,3n) = (n, l_{n}) + j(0,1)\end{equation}
for $j = 3n - l_{n} \in \N$. On the other hand,
\[\varphi(-n,n) = \varphi((-n,1) + (n-1)(0,1)) = \varphi(-n,1) + (n-1)(0,1).\]
Hence \[(n, l_{n}) + j(0,1) = \varphi(-n,1) + (n-1)(0,1).\]
It must therefore be that $ \varphi(-n,1) $ and $(n,l_{n})$ share the same first coordinate, and hence that $ \varphi(-n,1)  = (n,l_{n})$ and $j = n-1$. Now as
\[3n = l_{n} + j\]
from considering second coordinates in \eqref{-nnoneway}, it follows that $l_{n} = 2n+1$ for any $n \in \N$. As $2n+1 = n + (n+1)$, it must follow that $k = n+1$ is the unique index such that  $c_{k-1} < n \leq c_{k}$.
In particular, $c_n<n$ for all $n\in\N$, which is a contradiction, as the sequence $(c_k)$ grows at least exponentially by assumption \eqref{eq:sigma} from Construction \ref{cons:ssigma}.
\end{proof}

\begin{claim}
$\varphi((-\N)\times\N) = (-\N)\times\N$.
\end{claim}

\begin{proof}
Notice that every $(a, b) \in (-\N)\times\N$ can be decomposed as $(a, 1) + (b-1)(0, 1)$. Hence \begin{equation}\label{eq:varphiab} \varphi(a, b) = \varphi(a, 1) + (b-1)\varphi(0, 1) =\varphi(a, 1) + (b-1)(0, 1)\end{equation} by Claim \ref{cl:zxn1}. By Claims \ref{cl:zxn3} and \ref{cl:zxn5}, it follows that $\varphi(a,1) = (a,1)$, and hence \eqref{eq:varphiab} is equal to $(a,b)$. 
\end{proof}

Returning to the main proof of the theorem, it must be that $\varphi(S_\sigma \cap(\N_0\times\N)) = S_\tau \cap(\N_0\times\N)$ by the last claim above. We now apply Proposition \ref{prop:nonisomn0xn} and the result follows.
\end{proof}

\subsection{Intersection with \boldmath{$\Z \times \N_{0}$}} We now consider the case $U \times V = \Z \times \N_0$. 
\label{subs:4}
Each semigroup $S_\sigma \cap (\Z \times \N_0)$ is the union of lines $\{L_i : i \geq 1\}\cup\{(x, y) : x \leq 0, y \geq 0\}$.

In what follows we will repeatedly use the following observation, which follows immediately from Construction
\ref{cons:ssigma}:

\begin{lemma}
\label{la:alq0}
If
$(a,x)\in S_\sigma\cap (\Z\times\N_0)$ and $x\in\{0,1,2\}$ then $a\leq 0$. \qed
\end{lemma}

We proceed to describe the indecomposable elements of $S_\sigma \cap (\Z \times \N_0)$  in order to prove that there are uncountably many subdirect products of $\Z \times \N_0$.
Since $S_\sigma\cap (\Z\times\N_0)$ is a monoid, indecomposability will be understood to be in the monoid sense.

\begin{lemma}\label{lem:indeczxn0}
The set of indecomposable elements of $S_\sigma \cap (\Z\times\N_0)$ is exactly the set\[ \{(i,l_{i}) : i \geq 1 \}\cup\{(0, 0), (0, 1), (-1, 0)\}.\]
\end{lemma}

\begin{proof} As $(i, j) = (i, l_i) + (j - l_i)(0, 1)$ when $i \geq 1$ and $(i, j) = (-i)(-1, 0) + j(0, 1)$ for $i \leq 0$, we only need to prove that the elements of the set $\{(i,l_{i}) : i \geq 1 \}\cup\{(0, 0), (0, 1), (-1, 0)\}$ are all indecomposable.
For $(0,0)$, $(0,1)$, $(-1,0)$ this follows from Lemma \ref{la:alq0} and indecomposability of $0$ and $1$ in $\N_0$.

Now, consider an element $(i, l_i)$ for some $i > 0$ and suppose it has decomposition \[(i, l_i) = (a, x) + (b,y).\]
By \thref{lem:indecn0xn}, we cannot have $a, b \geq 0$. Suppose without loss of generality that $a < 0$, thus $b > i$. Hence, 
\[ l_i = x + y \geq y \geq l_b > l_i,\]
a contradiction. Therefore the elements $\{(i, l_i) : i \geq 1\}$ are indecomposable.\end{proof}

Now we can prove that $\Z\times\N_0$ has continuum many subdirect products up to isomorphism.

\begin{prop}
\label{prop:zxn0}
For any two sequences $\sigma$ and $\tau$ satisfying the conditions of  Construction \ref{cons:ssigma}, we have that 
\[\sigma \not = \tau \,\Rightarrow\,  S_{\sigma} \cap (\Z \times \N_0) \not \cong  S_{\tau} \cap (\Z \times \N_0).\]
Consequently, there are continuum many subdirect products of $\Z\times\N_0$ up to isomorphism.
\end{prop}

\begin{proof} 
Suppose that $\varphi: S_\sigma \cap (\Z\times\N_0) \rightarrow S_\tau \cap (\Z\times\N_0)$ is an isomorphism. We will proceed via a series of claims, aiming to show that $\varphi(S_{\sigma} \cap (\N_0 \times \N)) = S_{\tau} \cap (\N_0 \times \N)$, and then use Proposition \ref{prop:nonisomn0xn}.

\begin{claim}
\label{cl:zxn01}
$\varphi(0,0) = (0, 0)$.
\end{claim}

\begin{proof}
This follows from the fact that $(0, 0)$ is the unique identity element.
\end{proof}

\begin{claim}
\label{cl:zxn02}
$(0, 1) + (0, 1)$ is the unique decomposition of $(0,2)$.
\end{claim}

\begin{proof}
If $(0, 2) = (a, x) + (b, y)$ with $(a,x)+(b,y)\in S_\sigma\cap (\Z\cap\N_0)$ then $x,y\in\{0,1,2\}$,
and hence $a=b=0$ by Lemma \ref{la:alq0}, from which the claim follows readily.
\end{proof}

\begin{claim}
\label{cl:zxn03}
$(-1, 0) + (-1, 0)$ is the unique decomposition of $(-2, 0)$.
\end{claim}
\begin{proof}
Suppose that \[(-2, 0) = (a, x) + (b, y).\]
Since $x,y\in\N_0$, we must have $x=y=0$.

By Lemma \ref{la:alq0} we have $a,b\leq 0$, and the claim follows.
\end{proof}

\begin{claim}
\label{cl:zxn04}
$\varphi(0,1) = (0,1)$ and $\varphi(-1, 0) = (-1, 0)$.
\end{claim}
\begin{proof}
If we consider $(i, l_i)$ for some $i > 0$, then
\[(i, l_i) + (i, l_i) = (2i, 2l_i) = (2i, l_{2i}) + (2l_i - l_{2i})(0, 1).\]
This tells us that $(0,1)$ and $(-1, 0)$ are the unique indecomposables $x$ with the property that $2x = x + x$ is the only decomposition into a sum of indecomposables, and hence \[\varphi(\{(0,1),(-1,0)\})=\{(0,1),(-1,0)\}.\]
In particular, we can also deduce that $\varphi(\{(0,3),(-3,0)\})=\{(0,3),(-3,0)\}$. The element $(-3, 0)$ has a unique decomposition into a sum of indecomposables, namely $(3,0)= 3(-1, 0)$. By way of contrast, the element $(0,3)$ has more than one such decomposition, namely
\[(0,3) = 3(0,1) = (1,3) + (-1,0).\]
This is now sufficient to prove the claim, since it must therefore be that $\varphi(0,3) = (0,3)$ and $\varphi(-3,0) = (-3,0)$.\end{proof}

\begin{claim}
\label{cl:zxn05}
$\varphi((-\N) \times \N) = (-\N) \times \N$.
\end{claim}

\begin{proof}
Using Claim \ref{cl:zxn04}, for any $(i, j) \in (-\N) \times \N$ we have
\begin{align*}
\varphi(i, j) &= \varphi((-i)(-1, 0) + j (0, 1)) = (-i)\varphi(-1, 0) + j \varphi(0, 1)\\
&=(-i)(-1, 0) + j (0, 1)=(i,j),
\end{align*}
and the claim follows.
\end{proof}

Returning to the proof of the proposition, from Claim \ref{cl:zxn05} it follows that $\varphi(S_\sigma \cap (\N_0 \times \N)) = S_\tau \cap (\N_0 \times \N)$,  and the result follows by Proposition \ref{prop:nonisomn0xn}.
\end{proof}

\subsection{Intersection with \boldmath{$\Z \times \Z$}}
\label{subs:5}
We will now consider the final case where $U \times V = \Z \times \Z$. Notice that $S_{\sigma}\cap (\Z\times\Z)$ is simply just $S_\sigma$ from Construction \ref{cons:ssigma}.

We determine the indecomposable elements in $S_\sigma$ and use this to describe the isomorphisms between these semigroups.

\begin{lemma}\thlabel{lem:indeczxz}
The set of indecomposable elements of $S_\sigma$ is exactly the set\[ \{(c_k, c_k + k) : k \geq 2 \}\cup\{(0, 0), (0,1), (-1, -1)\}.\]
\end{lemma}

\begin{proof} 
Denote the above set by $I$.
First we consider $(i,j)\in S_\sigma\setminus I$ and show that it is decomposable.
Notice that
\[
(i,j)=(-i)(-1,-1)+(j-i)(0,1),
\]
and this is a non-trivial decomposition of $(i,j)$ in the following cases:
\begin{itemize}
\item
$i<-1$, because $j\geq i$, so that $-i\geq 2$;
\item
$i=-1$, because $j\geq 0$, so that $-i=1$ and $j-i>0$;
\item
$i=0$, because $-i=0$ and $j\geq 2$.
\end{itemize}
Now suppose $i>0$. Let $k \geq 2$ be the smallest index such that $i\leq c_k$.
Define $a:= c_k-i\geq 0$.
Since $(i,j)\in S_\sigma$ we have $j\geq i+k$, and we let
$b:=j-(i+k)= j-(c_{k}-a + k) \geq 0$.
Moreover, as $(i,j)\not\in I$ by assumption, we cannot have both $a=0$ and $b=0$. Hence
\[
(i, j) = (c_k - a, c_k - a +k + b) = (c_k, c_k + k) + a(-1, -1) + b (0, 1)
\]
is a non-trivial decomposition of $(i,j)$.

Now we show that each $(i,j)\in I$ is indecomposable.
We consider separately the cases where $(i,j)\in \{ (0,0),(0,1)\}$, $(i,j)=(-1,-1)$, and $(i,j)=(c_k,c_k+k)$.
In each of these cases we assume that
\[
(i,j)=(a,x)+(b,y) \quad \text{for some } (a,x),(b,y)\in S_\sigma\setminus\{(0,0)\}
\]
and proceed to derive a contradiction.

{\it Case 1: $(i,j)\in\{(0,0),(0,1)\}$.}
We cannot have $a = b = 0$ because $0, 1$ are both indecomposable in $\N_0$, and this would imply
that one of $(a,x)$, $(b,y)$ equals $(0,0)$. 
Now, without loss of generality suppose that $a < 0$, so that $x \geq a$. 
Then $b > 0$, so that $y \geq b +2$. Hence $1\geq j=x + y \geq a + b + 2 = 2$, a contradiction.

{\it Case 2: $(i,j)=(-1,-1)$.}
Suppose $a = -1$, $b = 0$, with $x \geq -1$, $y \geq 1$. Then $-1=j=x + y \geq -1 + 1 > -1$. 
Next, without loss of generality suppose that $a < -1$, $b > 0$. 
Reasoning as in the previous case,
$-1=j=x + y \geq a + b + 2 = 1$, a contradiction.

{\it Case 3: $(i,j)=(c_k,c_k+k)$.}
By \thref{lem:indecn0xn}, precisely one of $a,b$ must be negative, as $(c_{k},c_{k}+k)$ is indecomposable in 
$S_{\sigma} \cap (\N_{0} \times \N)$. Assume that $a < 0$ without loss. We have $b = c_k - a$ and $y = c_k + k -x$. As $x \geq a$, then $y \leq c_k+k-a$, i.e $y \leq b+k$. But as $b = c_{k}-a > c_{k}$, we know that $l_{b} \geq b+k+1$. This gives a contradiction, as if $(b,y) \in S_{\sigma}$, then 
\[
b+k \geq y > l_{b} \geq b+k+1.\qedhere
\]
\end{proof}

Now we can prove that $\Z\times\Z$ has continuum many subdirect products up to isomorphism.

\begin{prop}
\label{prop:zxz}
For any two sequences $\sigma$ and $\tau$ satisfying the conditions of  Construction \ref{cons:ssigma}, we have that 
\[\sigma \not = \tau \,\Rightarrow\,  S_{\sigma} \not \cong  S_{\tau}.\]
Consequently, there are continuum many (semigroup) subdirect products of $\Z\times\Z$ up to isomorphism.
\end{prop}

\begin{proof} 
Suppose that $\varphi: S_\sigma \rightarrow S_\tau$ is an isomorphism. We will proceed via a series of claims, aiming to show that $\sigma = \tau$.

\begin{claim}
$(-2,-2)$ has precisely one decomposition into a sum of non-zero indecomposables, namely
$(-2,-2)=(-1,-1)+(-1,-1)$.
\end{claim}

\begin{proof}
Suppose that
\[(-2,-2) = \sum_{i=1}^{n} (a_{i},x_{i}) \]
is a sum of $n\geq 2$ non-zero indecomposables.
 Consider the differences $d_{i} := x_{i}-a_{i}$. Noting that
\[ 
0 = -2 - (-2) =  \sum_{i=1}^{n}x_{i} - \sum_{i=1}^{n} a_{i} =  \sum_{i=1}^{n}d_{i},
\]
and that
 \[
 d_{i} = \begin{cases} 0 & \text{ if } (a_{i},x_{i}) = (-1,-1), \\  1 & \text{ if } (a_{i},x_{i}) = (0,1), \\  k & \text{ if } (a_{i},x_{i}) = (c_{k},c_{k}+k), \end{cases}
 \] 
 it follows that $(a_i,x_i)=(-1,-1)$ for all $i$, and the claim follows.
 \end{proof}

\begin{claim}
$(0,2)$ has precisely two decompositions into a sum of non-zero indecomposables, namely $(0,1)+(0,1)$ and $ (-1,-1) +(1,3)$.
\end{claim}

\begin{proof}
Suppose that 
\begin{equation}
(0,2) = \sum_{i=1}^{n} (a_{i},x_{i}), \label{diffsum}
\end{equation}
a sum of $n\geq 2$ non-zero indecomposables.
Let $d_{i}:=x_i-a_i$, as in the previous claim. 
This time, $\sum_{i=1}^{n}d_{i} = 2$. 
Hence the only choices for $(a_{i},x_{i})$ appearing in \eqref{diffsum} are $(-1,-1), (0,1)$ and $(c_{2},c_{2} +2) = (1,3)$.
Furthermore, there are either precisely two occurrences of $(0,1)$, or precisely one occurrence of $(1,3)$.
In the former case we obtain the decomposition $(0,2)=(0,1)+(0,1)$, and in the latter
$(0,2)=(1,2)+(-1,-1)$ as the only options. 
\end{proof}
 
\begin{claim}
For $k \geq 2$, the element $(2c_{k}, 2(c_{k} + k))$ has a decomposition into a sum of three or more non-zero indecomposables.
\end{claim}

\begin{proof}
Let $d:=c_{k+2}-2c_k$. Recalling \eqref{eq:sigma} from Construction \ref{cons:ssigma} we have
\[
d\geq 2c_{k+1}-2c_k\geq 4c_k-2c_k=2c_k\geq 2.
\]
Also let $n:=k-2\geq 0$.
Then
\begin{align*}
&\,(c_{k+2},c_{k+2}+k+2)+d(-1,-1)+n(0,1) \\
 =&
\,(c_{k+2}-d, c_{k+2}+k+2 -d +n)\\
=& \,\bigl(c_{k+2}- (c_{k+2}-2c_k), c_{k+2}+k+2 - (c_{k+2}-2c_k)+ k-2\bigr)\\
= & \,(2c_k, 2c_k+2k),
\end{align*}
a decomposition of $(2c_k, 2c_k+2k)$ into $1+d+n\geq 3$ non-zero indecomposables, as required.
\end{proof}

Having finished the series of claims, we can now proceed to directly prove that $\sigma = \tau$. As every element of $S_{\sigma}$ is a sum of the indecomposables described in \thref{lem:indeczxz}, we consider the images of these, each of which will be indecomposable in $S_{\tau}$. To distinguish between the sequences $\sigma$ and $\tau$, we will let $\sigma = (c_{k})_{k \geq 2}$, and $\tau = (C_{k})_{k \geq 2}$. 

Clearly $\varphi(0,0) = (0,0)$, being the identity of both monoids. By considering possible decompositions of $\varphi(2(-1,-1))$, $\varphi(2(0,1))$ and $\varphi(2(c_{k},c_{k}+k))$, Claims 1 to 3 assert that we must have $\varphi(-1,-1) = (-1,-1)$, $\varphi(0,1) = (0,1)$ and for any $k \geq 2$, $\varphi(c_{k}, c_{k}+k) = (C_{j}, C_{j} + j)$ for some $j \geq 2$.

Noting that
\[(c_{k}, c_{k}+k) + c_{k}(-1,-1) = k(0,1),\]
then applying $\varphi$ to the above shows that
\[(C_{j}, C_{j}+j) + c_{k}(-1,-1) = k(0,1),\]
and hence that $c_{k} = C_{j}$ and $k = j$. Therefore $c_{k} = C_{k}$ for all $k \geq 2$, and we conclude $\sigma =\tau$ as required. \end{proof}

\section{Concluding remarks}
\label{sec:crem}

As stated, our Main Theorem subsumes Theorem A and Theorem C for $k=2$ from  \cite{acnr1} dealing with $\N\times\N$.
However, within our proof we appeal to these results, rather than reprove them. In fact it is unclear whether our present methods could be modified to cover the $\N\times \N$ case.
For starters, the intersection $S_\sigma\cap (\N\times \N)$ is not subdirect: the elements $1$ and $2$ are missing from the first projection.
And secondly, our way of proving non-isomorphisms for different $\sigma$ via analysis of indecomposable elements would not work, as $S_\sigma\cap (\N\times \N)$ has a lot of indecomposables, which moreover depend on $\sigma$.

It is known that in groups, and more generally congruence permutable varieties, subdirect products of two factors coincide with the so called fiber products; this is known as Goursat's Lemma for groups
(see \cite[Theorem 4]{anderson09}) and Fleischer's Lemma in general (see \cite[Theorem 4.74]{mmt87}).
For two algebraic structures $A_1$, $A_2$ of the same type, a fiber product of $A_1$ and $A_2$ is a substructure of their direct product $A_1\times A_2$ of the form
$\{ (a_1,a_2)\in A_1\times A_2\::\: \varphi(a_1)=\varphi(a_2)\}$, where $\varphi_i : A_i\rightarrow Q$ ($i=1,2$) are \emph{onto} homomorphisms to a common quotient $Q$. 
It is well-known that Goursat's/Fleischer's Lemma does not extend to semigroups.
The Main Theorem offers a glimpse of just how badly it fails.
It is easy to see that each of $\N$, $\N_0$, $\Z$ has only countably many quotients (which are, respectively, all monogenic semigroups, all monogenic monoids, and all cyclic groups).
It therefore follows that for any $U,V\in \{\N,\N_0,\Z\}$ there are only countably many fiber products of $U$ and $V$.
Combining with the Main Theorem we conclude that uncountably many subdirect products of $U$ and $V$ are not fiber products.

The Main Theorem seems to suggest that it is rather hard for the direct product $U\times V$ of two infinite semigroups to contain only countably many subdirect products up to isomorphism.
But it is not impossible. One trivial example can be obtained by taking $U$ and $V$ to be two copies of an infinite zero semigroup $Z$ 
($zu=0$ for all $z,u\in Z$).
Then $Z\times Z$ is again a countable zero semigroup, i.e. $Z\times Z\cong Z$, as is every infinite subsemigroup of $Z\times Z$. Thus $Z$ is the only subdirect product of $Z\times Z$ up to isomorphism.
Another, less trivial, example is obtained by taking $U$ and $V$ to be two copies of a Tarski Monster $M$ -- an infinite simple group in which every proper subgroup has order $p$, where $p$ is a fixed prime \cite{ol81}.
Since $M\times M$ is periodic, its subsemigroups are in fact subgroups. Therefore semigroup subdirect products in $M\times M$ coincide with group subdirect products. And then it follows from simplicity and Goursat's Lemma that there are only two such subdirect products up to isomorphisms, namely $M$ and $M\times M$.

Motivated by the above discussion, we ask:

\begin{que}
Do there exist infinite non-periodic semigroups $U$ and $V$ such that $U\times V$ contains only countably many pairwise non-isomorphic subdirect products? Do there exist such $U$ and $V$ which are commutative?
\end{que}

We conjecture that the answer to the first question is affirmative and negative for the second.

\bibliographystyle{plain}

\end{document}